\documentclass[12pt, reqno]{amsart}
\usepackage[utf8]{inputenc}
\usepackage{amssymb,mathtools,cite,enumerate,color,eqnarray,hyperref,amsfonts,amsmath,amsthm,setspace,tikz,verbatim,charter,booktabs,multirow}
\usepackage[a4paper,margin=1.9cm,top=2.5cm,bottom=2.5cm,centering,vcentering]{geometry}
\numberwithin{equation}{section}

\definecolor{ao(english)}{rgb}{0.0, 0.5, 0.0}

\hypersetup{colorlinks=true, linkcolor=ao(english),citecolor=ao(english)}

\usepackage[normalem]{ulem}

\theoremstyle{plain}
\newtheorem{theorem}{Theorem}
\numberwithin{theorem}{section}
\newtheorem{corollary}[theorem]{Corollary}
\newtheorem*{corollary*}{Corollary}
\newtheorem*{Example*}{Example}

\newtheorem{lemma}[theorem]{Lemma}
\newtheorem{proposition}[theorem]{Proposition}

\newtheorem{conjecture}[theorem]{Conjecture}
\theoremstyle{definition}
\newtheorem{definition}[theorem]{Definition}
\newtheorem*{def*}{Definition}
\newtheorem*{theorem*}{Theorem}

\newtheorem*{definition*}{Definition}

\theoremstyle{remark}
\newtheorem*{remark}{Remark}

\allowdisplaybreaks
\begin{document}

\title[Families of congruences for the second order mock theta function $\mathcal{B}(q)$]{Infinite  families of congruences for the second order mock theta function $\mathcal{B}(q)$}

\author[H. Nath]{Hemjyoti Nath}
\address[H. Nath]{Department of Mathematics, University of Florida, P.O. Box 118105, Gainesville, FL 32611-8105, USA}
\email{h.nath@ufl.edu}

\author[H. Das]{Hirakjyoti Das}
\address[H. Das]{Department of Mathematics, B. Borooah College (Autonomous), Guwahati 781007, Assam, India}
\email{hirak@bborooahcollege.ac.in}




\begin{abstract}
The arithmetic properties of the second order mock theta function $\mathcal{B}(q)$, introduced by McIntosh, defined by
\begin{equation*}
    \mathcal{B}(q) := \sum_{n \geq 0} \frac{q^n (-q;q^2)_n}{(q;q^2)_{n+1}} = \sum_{n \geq 0}b(n)q^n,
\end{equation*}
have been extensively studied. For instance, for all $n\ge0$, Kaur and Rana \cite{kaur} established congruences such as for all $n\ge0$,
\begin{align*}
    b(12n+10) &\equiv 0 \pmod{36}, \quad
    b(18n+16) \equiv 0 \pmod{72},
\end{align*}
Chen and Mao \cite{CM} proved that for all $n\ge0$,
\begin{align*}
    b(4n+1) &\equiv 0 \pmod{2}, \quad
    b(4n+2) \equiv 0 \pmod{4},
\end{align*}
while Mao \cite{Mao} also showed that for all $n\ge0$,
\begin{align*}
    b(6n+2) &\equiv 0 \pmod{4}, \quad
    b(6n+4) \equiv 0 \pmod{9}.
\end{align*}
In this paper, we find new congruences and infinite families of congruences modulo $2, 4, 8, 36, 54, 72$ for the function $\mathcal{B}(q)$. For example, let $p \geq 5$ be a prime, if $\left(\frac{-3}{p}\right)_L = -1$, then for all $n, k \geq 0$ with $p \nmid n$, we have
\begin{equation*}
    b\left( 3p^{2k+1}n + \frac{p^{2k+2}-1}{2} \right) \equiv 0 \pmod{2}.
\end{equation*}
Let $p \geq 5$ be a prime and $1 \leq \ell \leq p - 1$ such that $\left( \frac{12\ell + 9}{p} \right)_L = -1$. Then for all $n, k \geq 0$, we have
\begin{equation*}
    b\left(6p^{2k+3}n + \frac{3p^{2k+2}(4\ell+3)-1}{2}\right) \equiv 0 \pmod{36}.
\end{equation*}
Our techniques involve elementary $q$-series and Maple.
\end{abstract}

\maketitle

\section{Introduction}

In his last letter to Hardy in 1920, Ramanujan introduced 17 $q$-series known as \emph{mock theta functions}, which resemble modular forms in their asymptotic behavior but do not transform like modular forms under the full modular group. The precise nature of mock theta functions remained unclear for decades. A major breakthrough came with the work of Zwegers \cite{zwegers}, who showed that mock theta functions can be completed by adding explicit non-holomorphic correction terms to form harmonic Maass forms of weight $1/2$. This reformulation places mock theta functions in the framework of harmonic Maass forms and connects them to the theory of modular forms, Maass forms, and automorphic representations.

In 2007, McIntosh investigated two second order mock theta functions, which are discussed in detail in \cite{McIntosh2007} and further elaborated in \cite{BringmannOnoRhoades}. These mock theta functions are defined by
\begin{align}
    \mathcal{A}(q) &=\sum_{n \geq 0} \frac{q^{(n+1)^2}(-q;q^2)_n}{(q;q^2)_{n+1}^2} = \sum_{n \geq 0} \frac{q^{n+1}(-q^2;q^2)_n}{(q;q^2)_{n+1}}, \label{1.1} \\
    \mathcal{B}(q) &=\sum_{n \geq 0} \frac{q^{n(n+1)}(-q^2;q^2)_n}{(q;q^2)_{n+1}^2} =  \sum_{n \geq 0} \frac{q^n(-q;q^2)_n}{(q;q^2)_{n+1}}, \label{1.2}
\end{align}
where the $q$-Pocchhamer symbols for complex numbers $a$ and $q$, $\mid q\mid <1$ are defined as
\begin{align*}
    (a,q)_n:=\prod_{j=0}^{n-1}(1-aq^j), \quad \quad (a,q)_\infty:=\prod_{j=0}^{\infty}(1-aq^j).
\end{align*}
The functions $\mathcal{A}(q)$ and $\mathcal{B}(q)$ have been interpreted combinatorially in terms of overpartitions via the odd Ferrers diagram, as described in \cite{burson}. In this paper, we investigate certain arithmetic properties of the second order mock theta function $\mathcal{B}(q)$. Bringmann, Ono, and Rhoades \cite{BringmannOnoRhoades} established that
\begin{equation}
    \frac{\mathcal{B}(q) + \mathcal{B}(-q)}{2} = \frac{f_4^5}{f_2^4}, \label{1.3}
\end{equation}
where $f_m^k := (q^m;q^m)^k_\infty$ for positive integers $m$ and $k$. For our work, we define the series
\begin{equation*}
    \mathcal{B}(q) := \sum_{n \geq 0} b(n)q^n. 
\end{equation*}
Following \eqref{1.3} and the above, the even part of \( \mathcal{B}(q) \) is given by
\begin{equation}
    \sum_{n \geq 0} b(2n)q^n = \frac{f_2^5}{f_1^4}. \label{Gen P_B 2n}
\end{equation}
In 2012, Chan and Mao \cite[(2.12), (2.13)]{CM} applied the theory of modular forms and utilized Zwegers’ work to establish two identities for \( b(n) \):
\begin{align}
    \sum_{n \geq 0} b(4n+1)q^n &= 2 \frac{f_2^{10}}{f_1^9}, \label{Gen P_B 4n+1}\\
    \sum_{n \geq 0} b(4n+2)q^n &= 4 \frac{f_2^2 f_4^4}{f_1^5}. \label{Gen P_B 4n+2}
\end{align}
In a subsequent work, Qu, Wang, and Yao \cite{quwangyao} demonstrated that all coefficients of the odd powers of \( q \) in \( \mathcal{B}(q) \) are even. Recently, Mao \cite[(1.5), (1.6)]{Mao} derived analogues of equations \eqref{Gen P_B 4n+1} and \eqref{Gen P_B 4n+2}:
\begin{align}
    \sum_{n \geq 0} b(6n+2)q^n = 4 \frac{f_2^{10} f_3^2}{f_1^{10} f_6}, \label{Gen P_B 6n+2}\\
    \sum_{n \geq 0} b(6n+4)q^n = 9 \frac{f_2^4 f_3^4 f_6}{f_1^8}. \label{Gen P_B 6n+4}
\end{align}
Also, in Mao's work \cite[(4.1)]{Mao}, the referee pointed out that
\begin{equation}
    \sum_{n \geq 0} b(4n)q^n = \frac{f_2^{14}}{f_1^9f_4^4}. \label{Gen P_B 4n}
\end{equation}
Recently, Kaur and Rana \cite{kaur}, proved several congruences for this function. In particular, they proved that for all $n\ge0$, 
\begin{align*}
    b(12n+9) &\equiv 0 \pmod{18}, \\
    b(12n+10) &\equiv 0 \pmod{36},\\
    b(18n+10) &\equiv 0 \pmod{72}, \\
    b(18n+16) &\equiv 0 \pmod{72}.
\end{align*}
Furthermore, they proved many other congruences modulo $5,10,72,144$.

Motivated by these works, we extend the list of congruences for the coefficients of $\mathcal{B}(q)$ by establishing infinite families of congruences for $b(n)$ modulo $2, 4, 8, 36, 54,$ and $72$.\\

Throughout the paper, let
\begin{equation}\label{n1}
    \sum_{n \geq 0}a(n)q^n := f_1^4f_2.
\end{equation}
For primes $p \geq 3$, we define $r$ and $s$ by
\begin{equation}\label{n4}
    r:=r(p) = a\left( \frac{p^2-1}{4} \right)+(-1)^{\frac{p-1}{2}}p\left( \frac{\frac{p^2-1}{4}}{p} \right)_L, \qquad s:=s(p)=p^3,
\end{equation}
where the Legendre symbol for a prime $p\geq3$ is defined as
\begin{align*}
\left(\dfrac{a}{p}\right)_L:=\begin{cases}\quad1,\quad \text{if $a$ is a quadratic residue modulo $p$ and $p\nmid a$,}\\\quad 0,\quad \text{if $p\mid a$,}\\~-1,\quad \text{if $a$ is a quadratic nonresidue modulo $p$.}
\end{cases}
\end{align*}
Then we also define
\begin{equation}\label{n2}
    \nu(p):=\begin{cases}
        2, & \text{if} \hspace{1.5mm} r \equiv 0 \pmod{8},\\
        4, & \text{if} \hspace{1.5mm} r \equiv 4 \pmod{8},\\
        8, & \text{if} \hspace{1.5mm} r \equiv 2 \pmod{4},
    \end{cases}
\end{equation}
\begin{equation}\label{n3}
    g(p):=\begin{cases}
        -s, & \text{if} \hspace{1.5mm} \nu(p) = 2,\\
        -r^2s+s^2, & \text{if} \hspace{1.5mm} \nu(p) = 4,\\
        -r^6s+5s^2r^4-6r^2s^3+s^4, & \text{if} \hspace{1.5mm} \nu(p) = 8.\\
    \end{cases}
\end{equation}

Our results are outlined below:

\begin{proposition}\label{Thm Gen P_B 3n}
    We have
    \begin{align}\label{Gen P_B 3n} 
    \sum_{n \geq 0} b(3n)q^n&= \frac{f_2^7f_3^2}{f_1^6f_4f_6}.
    \end{align}
\end{proposition}

\begin{remark}\label{t0.0.1.1}
    For all \( n \geq 0 \), we have
\begin{equation}
    b(3n) \equiv 
    \begin{cases} 
      (-1)^k \pmod{2}, & \text{if } n = 2k(3k+1) \text{ for some integer } k, \\
      0 \pmod{2}, & \text{otherwise.}
    \end{cases}
\end{equation}
\end{remark}
\noindent The above remark follows directly from Proposition \ref{Thm Gen P_B 3n}, \eqref{e2.0.3.4}, and \eqref{e0.1}. Hence, we do not include its proof.

\begin{theorem} \label{Thm Cong 1}
For all $n \geq 0$, we have  
    \begin{align}
        b\left( 6n+3 \right)&\equiv 0 \pmod{6}, \label{e501.0.0}\\
        b\left( 36n+22 \right)&\equiv 0 \pmod{36}, \label{e501.0.0.0}\\
        b\left( 12n+9 \right)&\equiv 0 \pmod{54}.\label{e50.0.0}
    \end{align}
\end{theorem}

\begin{theorem}\label{t0.0.1.0}
    Let $p \geq 5$ be a prime. If $\left(\frac{-3}{p}\right)_L=-1$, then for all $n, k \geq 0$ with $p \nmid n$, we have  
    \begin{equation}\label{e50.0}
       b\left( 3p^{2k+1}n + \frac{p^{2k+2}-1}{2} \right) \equiv 0 \pmod{2}.
    \end{equation}
\end{theorem}

\begin{theorem} \label{Thm 1.2}
Let $p$ be an odd prime and $1 \leq \ell \leq p - 1$ such that $\left( \frac{4\ell + 1}{p} \right)_L = -1$. Then for all $n \geq 0$, we have  
\begin{align} \label{Cor 2}
    b\left(2(pn + \ell)\right) &\equiv 0 \pmod{4}.
\end{align}
Furthermore, for all $n \geq 0 $ and $k \geq 0$, we have
\begin{equation} \label{Cor 2.1}
    b(2n) \equiv (-1)^{\frac{(k+1)(1-p)}{2}}p^{-k-1} b\left(2p^{2k+2}n + \frac{p^{2k+2}-1}{2}\right) \pmod{4}.
\end{equation}
\end{theorem}

Combining \eqref{Cor 2} and \eqref{Cor 2.1}, we arrive at the following corollary.

\begin{corollary}\label{C1}
    Let $p$ be an odd prime and $1 \leq \ell \leq p - 1$ such that $\left( \frac{4\ell + 1}{p} \right)_L = -1$. Then for all $n \geq 0$ and $k \geq 0$, we have  
\begin{align}\label{c1.1.1.1}
    b\left(2p^{2k+3}n+\frac{(4\ell+1)p^{2k+2}-1}{2}\right) &\equiv 0 \pmod{4}.
\end{align}
\end{corollary}

\begin{theorem} \label{Thm 1.1}
Let $p$ be an odd prime and $1 \leq \ell \leq p - 1$ such that $\left( \frac{8\ell + 1}{p} \right)_L = -1$. Then for all $n \geq 0$, we have  
\begin{align}\label{Cor 1}
    b\left(4(pn + \ell)\right) &\equiv 0 \pmod{8}.
\end{align}
Furthermore, for all $n \geq 0 $ and $k \geq 0$, we have
\begin{equation}\label{Cor 1.1}
    b(4n) \equiv b\left(4p^{2k+2}n + \frac{p^{2k+2}-1}{2}\right) \pmod{8}.
\end{equation}
\end{theorem}

\noindent\textbf{Example:} For $p = 5$, we have $\ell \in \{2, 4\}$. Then for all $n\ge0$, we have 
\begin{align*}
    b(20n + 8) &\equiv  b(20n + 16) \equiv 0 \pmod{8}.
\end{align*}
Also, from \cite[(3.8)]{kaur},
\begin{align*}
    b(20n + 8) &\equiv 
    b(20n + 16) \equiv 0 \pmod{5}.
\end{align*}
Therefore, by combining these above, we obtain  
\begin{align*}
    b(20n + 8) &\equiv 
    b(20n + 16) \equiv 0 \pmod{40}.
\end{align*}

As a result of \eqref{Cor 1} and \eqref{Cor 1.1}, we obtain the following corollary.

\begin{corollary}\label{C2}
    Let $p$ be an odd prime and $1 \leq \ell \leq p - 1$ such that $\left( \frac{8\ell + 1}{p} \right)_L = -1$. Then for all $n \geq 0$ and $k \geq 0$, we have  
\begin{align}\label{c1.1.1.2}
    b\left(4p^{2k+3}n+\frac{(8\ell +1)p^{2k+2}-1}{2}\right) &\equiv 0 \pmod{8}.
\end{align}
\end{corollary}

\begin{theorem}\label{Thm 1.3}
Let $p\geq 5$ be a prime and $1 \leq \ell \leq p - 1$ such that $\left( \frac{12\ell + 9}{p} \right)_L = -1$. Then for all $n \geq 0$, we have  
\begin{align} \label{Cor 3}
    b\left(6(pn + \ell)+4\right) &\equiv 0 \pmod{36}.
\end{align}
Furthermore, for all $n \geq 0 $ and $k \geq 0$, we have
\begin{equation} \label{Cor 3.1}
    b(6n+4) \equiv 9^{-k-1}(-1)^{\frac{(k+1)(1-p)}{2}}p^{-k-1} b\left(6p^{2k+2}n + \frac{9p^{2k+2}-1}{2}\right) \pmod{36}.
\end{equation}
\end{theorem}

The following corollary follows from \eqref{Cor 3} and \eqref{Cor 3.1}.

\begin{corollary}\label{cor last}
Let $p\geq 5$ be a prime and $1 \leq \ell \leq p - 1$ such that $\left( \frac{12\ell + 9}{p} \right)_L = -1$. Then for all $n \geq 0$ and $k \geq 0$, we have  
\begin{align}\label{c1.1.1.3}
    b\left(6p^{2k+3}n + \frac{3p^{2k+2}(4\ell+3)-1}{2}\right) &\equiv 0 \pmod{36}.
\end{align}
\end{corollary}

\begin{theorem}\label{Theorem f_1^7}
For all $n, \alpha \geq 0$, we have
    \begin{align}
        &b\left(12\left(5^{12\alpha + 11}\left(5n + \beta_1 - 4\right) + \frac{7(5^{12\alpha + 12} - 1)}{24}\right)+3\right) \nonumber \\
&\equiv b\left(12\left(7^{18\alpha + 17}\left(7n + \beta_2 - 6\right) + \frac{7(7^{18\alpha + 18} - 1)}{24}\right)+3\right) \nonumber \\
&\equiv b\left(12\left(13^{18\alpha + 17}\left(13n + \beta_3 - 12\right) + \frac{7(13^{18\alpha + 18} - 1)}{24}\right)+3\right) \equiv 0 \pmod{54}, \label{f1_3}
    \end{align}
where $0 \leq \beta_1 \leq 4$, $\beta_1 \neq 1$, $0 \leq \beta_2 \leq 6$, $\beta_2 \neq 2$, and $0 \leq \beta_3 \leq 12$, $\beta_3 \neq 3$. Moreover, for all $\alpha \geq 0$,
\begin{align}
b\left(12\left( \frac{295 \times 5^{12\alpha} - 7}{24} \right) +3\right) &\equiv b\left(12\left( \frac{295 \times 7^{18\alpha} - 7}{24} \right) +3\right) \nonumber\\
&\equiv b\left(12\left( \frac{295 \times 13^{18\alpha} - 7}{24} \right) +3\right)\nonumber\\
&\equiv 0 \pmod{54}, \label{f1_4}
\end{align}
\begin{align}\label{f1_5}
    b\left(12\left( \frac{7(5^{12\alpha} - 1)}{24} \right)+3\right) &\equiv b\left(12\left( \frac{7(13^{18\alpha} - 1)}{24} \right)+3\right) \equiv 6 \pmod{54},\\
    \label{f1_6}
b\left(12\left( \frac{7(7^{18\alpha} - 1)}{24} \right)+3\right) &\equiv 6\times(-1)^{\alpha} \pmod{54}.    
\end{align}
\end{theorem}

\begin{theorem}\label{nt1}
    Let $p \geq 3$ be a prime, $\nu(p)$ and $g(p)$ be as defined in \eqref{n2} and \eqref{n3}, respectively.
\begin{enumerate}[(i)]
  \item \textit{For all $n, k \geq 0$, if $p \nmid n$, then we have}
  \begin{equation}
    b\left( 18p^{2{\nu(p)(k+1)} -1}n +\frac{9p^{2{\nu(p)(k+1)}}-1}{2} \right) \equiv 0 \pmod{72}.\label{newm1}
  \end{equation}  
  \item \textit{For all $k \geq 0$, we have}
  \begin{equation}
    b\left( 18p^{2{\nu(p)k}} +\frac{9p^{2{\nu(p)k}}-1}{2} \right) \equiv 36g(p)^k \pmod{72}.\label{newm2}
  \end{equation}
\end{enumerate}
\end{theorem}
\noindent \textbf{Example.} Taking $p = 3$ in \eqref{newm1} and \eqref{newm2}, we find $r = 4$, $s = 27$, $\nu(3) = 4$ and $g(3) \equiv 1 \pmod{8}$. Thus, for all $n, k \geq 0$, if $3 \nmid n$, then
\begin{align*}
    b\left(18 \times 3^{8k+7}n + \frac{3^{8k+10}-1}{2} \right) & \equiv 0 \pmod{72},\\
    b\left(18 \times 3^{8k} + \frac{3^{8k+2}-1}{2} \right) & \equiv 36 \pmod{72}.
\end{align*}

\begin{theorem}\label{nt3}
   Let \( p \geq 3 \) be a prime and $\nu(p)$ be as defined in \eqref{n2}. For all \( n, k \geq 0 \), we have
    \begin{equation}
    b\left( 108p^{2{\nu(p)k}}n +\frac{153p^{2{\nu(p)k}}-1}{2} \right) \equiv 0 \pmod{72}.\label{newm2.0}
  \end{equation}
\end{theorem}

\begin{theorem}\label{nt2}
Suppose that $m$ is a positive integer such that $a(m) \equiv 0 \pmod{8}$, where $a(m)$ is defined in \eqref{n1}. Let $4m + 1 = \prod_{i=1}^{u} h_i \prod_{j=1}^{v} g_j^{\alpha_j}$ with each $\alpha_j \geq 2$ be the prime factorization of $4m + 1$. Then for all $n \geq 1$, we have
\begin{equation}\label{ez2}
b\left( 18mn^2+\frac{9n^2-1}{2} \right)\equiv 0 \pmod{72},
\end{equation}
where $\left(n, 2 \prod_{j=1}^{v} g_j^{\alpha_j} \right) = 1$.
\end{theorem}
\noindent \textbf{Example.} Let $m=31$. Then $a(31) \equiv 0 \pmod{8}$ and  $ 4 \cdot 31 + 1 = 125 = 5^3.$
Hence,  
$$
b\left( \frac{1125n^2 - 1}{2} \right) \equiv 0 \pmod{72},
$$  
where $(n, 250) = 1$.

\section{Preliminaries}\label{sec:prelim}
In this section, we gather several results from elementary $q$-series analysis that will be useful in proving our main results. First, the well known Ramanujan's general theta function $f(a,b)$ \cite[(1.2.1)]{Spirit} is defined by
\begin{equation*}
    f(a,b)=\sum_{n \in \mathbb{Z}}a^{n(n+1)/2}b^{n(n-1)/2}, \quad |ab|<1.
\end{equation*}
Three special cases of $f(a,b)$ are the theta functions $\varphi(q)$, $\psi(q)$ and $f(-q)$, which are given by:
\begin{align*}
    \varphi(q) &:= f(q,q) = \sum_{n \in \mathbb{Z}}q^{n^2} = (-q;q^2)_{\infty}^2(q^2;q^2)_{\infty} = \frac{f_2^5}{f_1^2f_4^2},\\
    \psi(q) &:= f(q,q^3) = \sum_{n \in \mathbb{Z}}q^{n(n+1)/2} = \frac{(q^2;q^2)_{\infty}}{(q;q^2)_{\infty}} = \frac{f_2^2}{f_1},\\
    f(-q)&:=f(-q,-q^2)= \sum_{n \in \mathbb{Z}}(-1)^{n}q^{n(3n-1)/2} = (q;q)_{\infty}=f_1.
\end{align*}
In terms of $f(a,b)$, Jacobi's triple product identity \cite[(1.3.11)]{Spirit} is given by
\begin{equation*}
    f(a,b) = (-a;ab)_{\infty}(-b;ab)_{\infty}(ab;ab)_{\infty}.
\end{equation*}
We need Euler's pentagonal number theorem \cite[Eq. (1.3.18)]{Spirit}:
\begin{equation}\label{e2.0.3.4}
    f_1=\sum_{n \in \mathbb{Z}}(-1)^nq^{n(3n-1)/2}
\end{equation}
and Jacobi's Triple Product identity \cite[Eq. (1.3.24)]{Spirit}:
\begin{equation}\label{e2.0.3.3}
    f_1^3=\sum_{n\geq 0}(-1)^n(2n+1)q^{n(n+1)/2}.
\end{equation}

In order to prove Theorem \ref{Theorem f_1^7}, we need the following:

\begin{definition}
    Let \( m \geq 2 \) be an integer and   $$\pi := \left\{ \{S_1(n)\}_{n=0}^{\infty}, \{S_2(n)\}_{n=0}^{\infty}, \ldots, \{S_\ell(n)\}_{n=0}^{\infty} \right\}$$
be a system of sequences, where \( \{S_i(n)\}_{n=0}^{\infty} \) (\(1 \leq i \leq \ell\)) are integer sequences. We say that \( R_m(\pi) \) is the rank of the system \( \pi \) modulo \( m \) if \( R_m(\pi) \) is the smallest integer such that
$$S_1(R_m(\pi)) \equiv S_2(R_m(\pi)) \equiv \cdots \equiv S_\ell(R_m(\pi)) \equiv 0 \pmod{m}.$$
For instance, set \( m = 2 \) and define
\begin{align*}
    F_1(n+2) & := 2F_1(n+1) + F_1(n),\\
    F_2(n+2) &:= F_2(n+1) - 3F_2(n),
\end{align*}
$n \geq 0$, with \( F_1(0) = F_2(1) = 0 \) and \( F_2(0) = F_1(1) = 1 \). It is easy to check that \( \{F_1(n)\}_{n=0}^{\infty} = \{0, 1, 2, 5, 12, \ldots\} \) and \( \{F_2(n)\}_{n=0}^{\infty} = \{1, 0, -3, -3, 6, \ldots\} \). Therefore, the rank of the system \( \{F_1(n)_{n=0}^\infty, F_2(n)_{n=}^\infty\} \) modulo 2 is 4, namely, 
$$R_2(\{F_1(n)_{n=0}^\infty, F_2(n)_{n=}^\infty\}) = 4.$$
As usual, when \( \pi = \emptyset \), we define \( R_m(\emptyset) = 0 \).

It should be noted that for an integer sequence \( \{S(n)\}_{n=0}^{\infty} \) and for an integer \( m \geq 2 \), there may not exist an integer \( r \) such that \( S(r) \equiv 0 \pmod{m} \). For example, set \( m = 2 \) and \( \{S(n)\}_{n=0}^{\infty} = \{2n + 1\}_{n=0}^{\infty} \), we see that for any \( j \), \( S(j) \not\equiv 0 \pmod{2} \).
\end{definition}

We need a few recurrent sequences. Let $A_k(\alpha)$ and $B_{k,i}(\alpha)$ satisfy the same linear recurrence relation:
\begin{equation*}
     A_k(\alpha) = x(k)A_k(\alpha-1) + y(k)A_k(\alpha-2),
\end{equation*}
$C_k(\alpha)$ and $D_{k,i}(\alpha)$ satisfy the same linear recurrence relation:
\begin{equation*}
    C_k(\alpha) = s(k)C_k(\alpha-1) + t(k)C_k(\alpha-2),
\end{equation*}
$E_k(\alpha)$ and $F_{k,i}(\alpha)$ satisfy the same linear recurrence relation:
\begin{equation*}
    E_k(\alpha) = u(k)E_k(\alpha-1) + v(k)E_k(\alpha-2),
\end{equation*}
where the values of $x(k), y(k), A_k(0), A_k(1), B_{k,i}(0), B_{k,i}(1)$; $s(k), t(k), C_k(0), C_k(1), D_{k,i}(0), D_{k,i}(1)$; and $u(k), v(k), E_k(0), E_k(1), E_{k,i}(0), E_{k,i}(1)$ can be found in Appendix I, Appendix II, and Appendix III, respectively in \cite{Xia1}.

Let $\{B_{k,i}(\alpha)\}$, $\{D_{k,i}(\alpha)\}$ and $\{F_{k,i}(\alpha)\}$ be defined as above. For $1 \leq k \leq 24$, we always assume that $r_1(k,m)$, $r_2(k,m)$ and $r_3(k,m)$ are the ranks of $\{B_{k,i}(\alpha)\}$, $\{D_{k,i}(\alpha)\}$ and $\{F_{k,i}(\alpha)\}$ modulo $m$ respectively, namely,
$$R_m(\{B_{k,i}(\alpha)\}) = r_1(k,m), \quad R_m(\{D_{k,i}(\alpha)\}) = r_2(k,m), \quad R_m(\{F_{k,i}(\alpha)\}) = r_3(k,m).$$

Suppose that $\displaystyle{\sum_{n \geq 0}} a_k(n)q^n = f_1^k$. Then, the following result holds, as given in \cite[Corollary 1.4]{Xia1}.

\begin{corollary}
Let $m \geq 2$, then for all $n$ and $\alpha$, we have
\begin{align}
a_k \left( 5^{2(r_1(k,m)+1)(\alpha+1)-1} \left(5n + \beta_1 - \left\lfloor \frac{k}{5} \right\rfloor \right) + \frac{k(5^{2(r_1(k,m)+1)(\alpha+1)} - 1)}{24} \right) &\equiv 0 \pmod{m}, \label{f_1 7 1}\\
a_k \left( 7^{2(r_2(k,m)+1)(\alpha+1)-1} \left(7n + \beta_2 - \left\lfloor \frac{2k}{7} \right\rfloor \right) + \frac{k(7^{2(r_2(k,m)+1)(\alpha+1)} - 1)}{24} \right) &\equiv 0 \pmod{m}, \label{f_1 7 2}\\
a_k \left( 13^{2(r_3(k,m)+1)(\alpha+1)-1} \left(13n + \beta_3 - \left\lfloor \frac{7k}{13} \right\rfloor \right) + \frac{k(13^{2(r_3(k,m)+1)(\alpha+1)} - 1)}{24} \right) &\equiv 0 \pmod{m}. \label{f_1 7 3}
\end{align}
where $0 \leq \beta_1 \leq 4$ with $\beta_1 \neq \left\lfloor \frac{k}{5} \right\rfloor$, $0 \leq \beta_2 \leq 6$ with $\beta_2 \neq \left\lfloor \frac{2k}{7} \right\rfloor$ and $0 \leq \beta_3 \leq 12$ with $\beta_3 \neq \left\lfloor \frac{7k}{13} \right\rfloor$. If $a_k(n_0) \equiv 0 \pmod{m}$, then for all $\alpha \geq 0$,
\begin{align}
     a_k \left( \frac{(24n_0 + k)5^{2(r_1(k,m)+1)\alpha} - k}{24} \right)  &\equiv a_k \left( \frac{(24n_0 + k)7^{2(r_2(k,m)+1)\alpha} - k}{24} \right) \nonumber\\
    & \equiv a_k \left( \frac{(24n_0 + k)13^{2(r_3(k,m)+1)\alpha} - k}{24} \right) \equiv 0 \pmod{m}. \label{f_1 7 4}
\end{align}
Moreover, for all $\alpha \geq 0$,
\begin{align}
    a_k\left( \frac{k(5^{2(r_1(k,m)+1)\alpha} - 1)}{24} \right) & \equiv [A_k(r_1(k,m))]^{\alpha} \pmod{m}, \label{f_1 7 5}\\
    a_k\left( \frac{k(7^{2(r_2(k,m)+1)\alpha} - 1)}{24} \right) & \equiv [C_k(r_2(k,m))]^{\alpha} \pmod{m}, \label{f_1 7 6}\\
    a_k\left( \frac{k(13^{2(r_3(k,m)+1)\alpha} - 1)}{24} \right) & \equiv [E_k(r_3(k,m))]^{\alpha} \pmod{m}. \label{f_1 7 7}
\end{align}

\end{corollary}

The table below shows the initial values required for the proof of Theorem \ref{Theorem f_1^7}.

\begin{table}[h!]
\centering
\small 
\begin{tabular}{|c|l|l|l|}
\hline
\textbf{Dataset} & \textbf{Parameters} & \textbf{Main Values} & \textbf{Subcomponents} \\
\hline
\textbf{$A$-Series} 
& \( x(7) = 66, \) 
& \( A_7(0) = 125, \) 
& \( B_{7,1}(0) = 14,\) \\
 
& \( y(7) = -5^5 \) 
& \( A_7(1) = 5125 \) 
& \(  B_{7,1}(1) = 924 \) \\
\hline
\multirow{2}{*}{\textbf{$C$-Series}} 
& \multirow{2}{*}{\( s(7) = -176, \)} 
& \multirow{2}{*}{\( C_7(0) = 0,\)} 
& \( D_{7,1}(0) = 49, \) \\

& & & \(  D_{7,1}(1) = -8624, \) \\
\multirow{2}{*}{\textbf{}} 
& \multirow{2}{*}{\( t(7) = -7^5 \)} 
& \multirow{2}{*}{\(C_7(1) = -7^5 \)} 
& \( D_{7,2}(0) = 1,\) \\

& & & \( D_{7,2}(1) = -176 \) \\
\hline
\multirow{3}{*}{\textbf{$E$-Series}} 
& \multirow{3}{*}{\( u(7) = -658, \)} 
& \multirow{3}{*}{\( E_7(0) = -13^3,\)} 
& \( F_{7,1}(0) = 0, \) \\

& & & \(  F_{7,1}(1) = 0, \) \\

& & & \( F_{7,2}(0) = 364, \) \\
\multirow{3}{*}{\textbf{}} 
& \multirow{3}{*}{\(  v(7) = -13^5 \)} 
& \multirow{3}{*}{\( E_7(1) = 1074333 \)} 
&  \(  F_{7,2}(1) = -239512,\)\\

& & & \( F_{7,3}(0) = 98, \) \\

& & & \(  F_{7,3}(1) = -64484 \) \\
\hline
\end{tabular}
\caption{}
\label{table1}
\end{table}

The following identity holds for all \( t \geq 1 \):
\begin{align}\label{e0.1}
    f_t^2 \equiv f_{2t} \pmod{2}.
\end{align}

\begin{lemma} \cite{hir}
We have  
\begin{align}
    \frac{f_3}{f_1^3} &= \frac{f_4^6 f_6^3}{f_2^9 f_{12}^2} + 3q \frac{f_4^2 f_6 f_{12}^2}{f_2^7}, \label{l1}\\
    \frac{f_3^2}{f_1^2} &= \frac{f_4^4 f_6 f_{12}^2}{f_2^5 f_8 f_{24}} + 2q \frac{f_4 f_6^2 f_8 f_{24}}{f_2^4 f_{12}},\label{l3}\\
    \frac{f_2}{f_1^2}&= \frac{f_6^4f_9^6}{f_3^8f_{18}^3} + 2q\frac{f_6^3f_9^3}{f_3^7}+4q^2\frac{f_6^2f_{18}^3}{f_3^6}.\label{l4}
\end{align}
\end{lemma}

\begin{lemma}\cite[Lemma 2.3]{AN}
For any prime $p \geq 3$, we have
\begin{align}\label{f_1^3 p d}
    f(-q)^3 = f_1^3 &= \frac{1}{2} \sum_{\substack{k=0 \\ k \neq \frac{p-1}{2}}}^{p-1} (-1)^k q^{\frac{k^2 + k}{2}} \sum_{n=-\infty}^{\infty} (-1)^n (2pn + 2k + 1) q^{\frac{pn(pn + 2k + 1)}{2}}\notag\\
    &\quad+ (-1)^{\frac{p-1}{2}} pq^{\frac{p^2 - 1}{8}} f(-q^{p^2})^3.
\end{align}
Furthermore, if $0 \leq k \leq p-1$ and $k \neq \frac{p-1}{2}$, then $\frac{k^2 + k}{2} \not\equiv \frac{p^2 - 1}{8} \pmod{p}$.
\end{lemma}

\begin{lemma}
For any odd prime $p$, the $p$-dissection of $\psi(q)$ is given by
\begin{equation}\label{psi p d}
    \psi(q) = \sum_{k=0}^{\frac{p-3}{2}} q^{\frac{k^2 + k}{2}}f\left( \frac{p^2 + (2k+1)p}{2}, \frac{p^2 - (2k+1)p}{2} \right) + q^{\frac{p^2 - 1}{8}} \psi(q^{p^2}).
\end{equation}
Furthermore,
\[\frac{k^2 + k}{2} \not\equiv \frac{p^2 - 1}{8} \pmod{p} \quad \text{for } 0 \leq k \leq \frac{p-3}{2}.\]
\end{lemma}

The paper is organized as follows. In Section \ref{sec 1}, we present a validation for Proposition \ref{Thm Gen P_B 3n} using Maple. In Section \ref{sec 2}, we provide an elementary proof of Theorem \ref{Thm Cong 1}. In Sections \ref{sec 2}--\ref{sec 6}, we prove Theorems \ref{t0.0.1.0}--\ref{nt2}. Finally, in Section \ref{sec 7}, we conclude the paper with some remarks.

\section{Computer Verification of Proposition \ref{Thm Gen P_B 3n} with Maple}\label{sec 1}
An analytic proof of \eqref{Gen P_B 3n} is desirable but seems quite challenging. Therefore, we find \eqref{Gen P_B 3n} with Maple with the help of the \texttt{$q$-series} package developed by Garvan \cite{Gar99}. The Maple commands are written below:
\begin{align*}
    &>\,\texttt{with(qseries):}\\
    &>\,\texttt{with(ETA):}\\
    &>\, \texttt{M}\coloneqq \texttt{n}\rightarrow \frac{q^{n\cdot \left(n+1\right)}\texttt{product}\left(1+q^{2\cdot j+2},j=0..n-1\right)}{\texttt{product}\left(1-q^{2\cdot j+1},j=0..n\right)^{2}};\\
    &\quad\quad \textcolor{blue}{M\coloneqq n\,\mapsto \dfrac{q^{n\cdot (n+1)}\cdot\left(\displaystyle{\prod_{j=0}^{n-1}}(1+q^{2\cdot j+2}) \right)}{\left(\displaystyle{\prod_{j=0}^{n}}(1-q^{2\cdot j+1}) \right)^2}}\\
    &>\, \texttt{add(coeff(series(add($M(n)$, $n$ = 0..3001), $q$, 3001), $q$,$3n$)$q^{n}$,$n$ = 0..1000);}\\
    &\quad \quad \textcolor{blue}{1+6 q+20 q^{2}+54 q^{3}+129 q^{4}+282 q^{5}+580 q^{6}+1134 q^{7}+2129 q^{8}+3864 q^{9}+6812 q^{10}}\\
    &\quad \quad\textcolor{blue}{+\cdots +213349810381656918950139080761530838997673680728800 q^{999}}\\
    &\quad\quad \textcolor{blue}{+226621172121394950726209294107119502848934920312640 q^{1000}}\\
    &>\, H:=1+6 q+20 q^{2}+54 q^{3}+129 q^{4}+282 q^{5}+580 q^{6}+1134 q^{7}+2129 q^{8}+3864 q^{9}\\
    & \quad \,+6812 q^{10}+\cdots +213349810381656918950139080761530838997673680728800 q^{999}\\
    &\quad\,+ 226621172121394950726209294107119502848934920312640 q^{1000}:\\
    &>\, \texttt{qetamake($H$, $q$, 1000);}\\
    &\quad\quad \textcolor{blue}{\dfrac{f_2^7f_3^2}{f_1^6f_4f_6}}
\end{align*}
The instructions for the installations of the packages \texttt{$q$-series} and \texttt{ETA} are given by Garvan \cite{Gar99}. It is also worth noting that it is actually enough to obtain $H$ above up to the sixth power of $q$ as the \texttt{qetamake()} command gives \eqref{Gen P_B 3n} for $H$ up to all exponents of $q$ beyond 6. 

Thus, we validate Proposition \ref{Thm Gen P_B 3n}.\qed

The proofs of the remaining theorems and corollaries are given by $q$-series.

\section{Proof of Theorem \ref{Thm Cong 1}}\label{sec 2}

We perform a $2$-dissection of \eqref{Gen P_B 3n} using \eqref{l1}, yielding
    \begin{equation*}
        \sum_{n \geq 0}  b(3n) q^n = \frac{f_4^{11} f_6^5}{f_2^{11} f_{12}^4} + 6q \frac{f_4^7 f_6^3}{f_2^9} + 9q^2 \frac{f_4^3 f_6 f_{12}^4}{f_2^7}.
    \end{equation*}
    Extracting the terms containing $q^{2n+1}$ from both sides, we obtain  
    \begin{equation}\label{f1}
        \sum_{n \geq 0}  b(6n+3) q^n = 6\frac{f_2^7 f_3^3}{f_1^9}.
    \end{equation}
It follows from \eqref{f1} that \eqref{e501.0.0} is true. \\

Under modulo $4$, \eqref{Gen P_B 6n+4} becomes
\begin{equation*}
    \sum_{n \geq 0} b(6n+4)q^n \equiv 9f_6^3\left( \frac{f_4}{f_2^2} \right)^2 \pmod{36}.
\end{equation*}
Next, replacing $q$ with $q^2$ in \eqref{l4}, substituting it into the above expression, and then extracting the terms involving $q^{3n}$ from both sides, we obtain
\begin{equation*}
    \sum_{n \geq 0} b(18n+4)q^n \equiv 9\frac{f_2f_4^2f_6^2}{f_{12}} \pmod{36}.
\end{equation*}
Finally, extracting the terms involving $q^{2n+1}$ from both sides, we arrive at \eqref{e501.0.0.0}.\\

Performing a further 2-dissection on \eqref{f1} using \eqref{l1}, we obtain
    \begin{equation*}
        \sum_{n \geq 0}  b(6n+3) q^n = 6 \frac{f_4^{18} f_6^{9}}{f_2^{20} f_{12}^6} + 54q \frac{f_4^{14} f_6^7}{f_2^{18} f_{12}^2} + 162q^2 \frac{f_4^{10} f_6^5 f_{12}^2}{f_{2}^{16}} + 162q^3 \frac{f_4^6 f_6^3 f_{12}^6}{f_2^{14}}.
    \end{equation*}
    Extracting the terms containing $q^{2n+1}$ from both sides, we arrive at  
    \begin{equation}\label{mod 54}
        \sum_{n \geq 0}  b(12n+9) q^n = 54\left( \frac{f_2^{14} f_3^7}{f_1^{18} f_6^2} + 3q \frac{f_2^6 f_3^3 f_{6}^6}{f_1^{14}} \right).
    \end{equation}
    This establishes \eqref{e50.0.0}. \qed

\section{Proof of Theorem \ref{t0.0.1.0}}\label{sec 3}

We begin with the identity  \eqref{Gen P_B 3n}:
\begin{equation}\label{et1.1}
    \sum_{n \geq 0}  b(3n) q^n = \frac{f_2^7f_3^2}{f_1^6f_4f_6}.
\end{equation}
Applying \eqref{e0.1} in \eqref{et1.1}, we obtain  
\begin{equation}\label{et3.1}
     \sum_{n \geq 0}  b(3n) q^n \equiv f_1^4 = f_1f_1^3 \pmod{2}.
\end{equation}
Combining \eqref{e2.0.3.3}, \eqref{e2.0.3.4} and \eqref{et3.1}, we obtain  
\begin{equation}\label{e50.3.0}
    \sum_{n \geq 0}  b(3n) q^{24n+4} \equiv \sum_{m \in \mathbb{Z}}\sum_{j \in \mathbb{Z}} (-1)^{m+j}(2m+1)q^{3(2m+1)^2+(6j-1)^2} \pmod{2}.
\end{equation}
From \eqref{e50.3.0}, we conclude that if $24n+4$ is not of the form \( 3(2m+1)^2+(6j-1)^2 \), then  
\begin{equation*}
    b(3n) \equiv 0 \pmod{2}.
\end{equation*}
Now, let $p \geq 5$ be a prime with $\left(\frac{-3}{p}\right)_L=-1$, and let $\nu_p(N)$ denote the exponent of the highest power of $p$ dividing $N$. If $N$ is of the form \( 3x^2+y^2 \), then $\nu_p(N)$ is even. However, it is easy to verify that if $p \nmid n$, then  
\begin{equation*}
    \nu_p\left( 24 \left( p^{2k+1}n+ \frac{p^{2k+2}-1}{6} \right)+ 4 \right) = 2k+1.
\end{equation*}
Thus, \( 24 \left( p^{2k+1}n+ \frac{p^{2k+2}-1}{6} \right)+4 \) is not of the form \( 3x^2+y^2 \) when $p \nmid n$. Hence, \eqref{e50.0} holds. \qed

\section{Proofs of Theorems \ref{Thm 1.2}, \ref{Thm 1.1}, and \ref{Thm 1.3}}\label{sec 4}

\begin{proof}[Proof of Theorem \ref{Thm 1.2}]
Reducing \eqref{Gen P_B 2n} modulo $4$, we have
\begin{align*}
    \sum_{n \geq 0} b(2n)q^n & \equiv f_2^3 \pmod{4}.
\end{align*}
Now using \eqref{f_1^3 p d}, we get 
\begin{align}\label{mod 8.1}
    \sum_{n \geq 0} b(2n)q^n & \equiv  \frac{1}{2} \sum_{\substack{k=0 \\ k \neq \frac{p-1}{2}}}^{p-1} (-1)^k q^{k^2+k} \sum_{n=-\infty}^{\infty}(-1)^n (2pn + 2k + 1) q^{pn(pn + 2k + 1)} \nonumber\\
    &+ (-1)^{\frac{p-1}{2}} pq^{\frac{p^2 - 1}{4}} f(-q^{2p^2})^3 \pmod{4}.
\end{align}
Since $\ell$ was chosen such that $4\ell + 1$ is a quadratic nonresidue modulo $p$, it follows that for any integer $k$,
\[
4pn + 4\ell + 1 \neq (2k + 1)^2.
\]
This implies that $pn + \ell$ cannot be represented as $k^2+k$ for any integer $k$. Also, since $4\ell + 1$ is a quadratic nonresidue modulo $p$, we have
\[
4pn + 4\ell + 1 \neq p^2.
\]
Hence, $pn + \ell$ cannot be equal to $\frac{p^2 - 1}{4}$ either. Therefore, upon extracting the terms involving $q^{pn + \ell}$, we obtain
\begin{equation}\label{c1.1.1}
    b(2(pn + \ell)) \equiv 0 \pmod{4}.
\end{equation}

Further, extracting the terms involving $q^{p^2n+(p^2-1)/4}$ from both sides of \eqref{mod 8.1}, we get
\begin{align}\label{mod 8 4}
    b\left( 2 \left(p^2n+\frac{p^2-1}{4}\right) \right) & \equiv (-1)^{\frac{p-1}{2}} p f_2^3 \pmod{4},
\end{align}
which is equivalent to
\begin{align}\label{mod 8 5}
     b(2n) & \equiv (-1)^{\frac{1-p}{2}}p^{-1}b\left( 2 \left(p^{2}n+\frac{p^2-1}{4}\right) \right) \pmod{4}.
\end{align}
Successive iterations on \eqref{mod 8 5} give
\begin{align}
    b(2n) &\equiv (-1)^{\frac{1-p}{2}}p^{-1} b\left( 2 \left(p^2n+\frac{p^2-1}{4}\right) \right) \nonumber\\
           &\equiv (-1)^{{1-p}}p^{-2} b\left( 2 \left(p^2\left(p^2n+\frac{p^2-1}{4}\right)+\frac{p^2-1}{4}\right) \right)\nonumber\\
           & \,\,\,\vdots\nonumber\\
           &\equiv (-1)^{\frac{(k+1)(1-p)}{2}}p^{-k-1} b\left( 2\left( p^{2k+2}n + \frac{p^2-1}{4}\left( p^{2k}+ \cdots +p^2+1 \right) \right) \right)\nonumber\\
           & \equiv (-1)^{\frac{(k+1)(1-p)}{2}}p^{-k-1} b\left(2\left(p^{2k+2}n + \frac{p^{2k+2}-1}{4}\right)\right)\nonumber\\
           & \equiv (-1)^{\frac{(k+1)(1-p)}{2}}p^{-k-1} b\left(2p^{2k+2}n + \frac{p^{2k+2}-1}{2}\right) \pmod{4} \label{c1.1}.
\end{align}
This completes the proof.
\end{proof}

\begin{proof}[Proof of Corollary \ref{C1}]
Substituting $pn+\ell$ into \eqref{c1.1} and then applying \eqref{c1.1.1}, we obtain \eqref{c1.1.1.1}.
\end{proof}

\begin{proof}[Proof of Theorem \ref{Thm 1.1}]
Reducing \eqref{Gen P_B 4n} modulo $8$, we have
\begin{align*}
    \sum_{n \geq 0} b(4n)q^n & \equiv \frac{f_2^2}{f_1} = \psi(q) \pmod{8}.
\end{align*}
Now using \eqref{psi p d}, we get 
\begin{align}\label{mod 8}
    \sum_{n \geq 0} b(4n)q^n & \equiv \sum_{k=0}^{\frac{p-3}{2}} q^{\frac{k^2 + k}{2}}f\left( \frac{p^2 + (2k+1)p}{2}, \frac{p^2 - (2k+1)p}{2} \right) + q^{\frac{p^2 - 1}{8}} \psi(q^{p^2}) \pmod{8}.
\end{align}
Since $\ell$ was chosen such that $8\ell + 1$ is a quadratic nonresidue modulo $p$, it follows that for any integer $k$,
\[
8pn + 8\ell + 1 \neq (2k + 1)^2.
\]
This implies that $pn + \ell$ cannot be represented as $\frac{k^2 + k}{2}$ for any integer $k$. Also, since $8\ell + 1$ is a quadratic nonresidue modulo $p$, we have
\[
8pn + 8\ell + 1 \neq p^2.
\]
Hence, $pn + \ell$ cannot be equal to $\frac{p^2 - 1}{8}$ either. Therefore, upon extracting the terms involving $q^{pn + \ell}$, we obtain
\begin{align}\label{c1.1.2}
    b(4(pn + \ell)) \equiv 0 \pmod{8}.
\end{align}

Further, extracting the terms involving $q^{p^2n+(p^2-1)/8}$ from both sides of \eqref{mod 8}, we get
\begin{align}\label{mod 8 2}
    b\left( 4 \left(p^2n+\frac{p^2-1}{8}\right) \right) & \equiv \psi(q) \pmod{8},
\end{align}
which is equivalent to
\begin{align}\label{mod 8 3}
     b(4n) & \equiv b\left( 4 \left(p^2n+\frac{p^2-1}{8}\right) \right) \pmod{8}.
\end{align}
Successive iterations on \eqref{mod 8 3} give
\begin{align}
    b(4n) &\equiv b\left( 4 \left(p^2n+\frac{p^2-1}{8}\right) \right) \nonumber\\
           &\equiv b\left( 4 \left(p^2\left(p^2n+\frac{p^2-1}{8}\right)+\frac{p^2-1}{8}\right) \right)\nonumber\\
           & \,\,\,\vdots\nonumber\\
           & \equiv b\left( 4\left( p^{2k+2}n + \frac{p^2-1}{8}\left( p^{2k}+p^{2k-1}+ \cdots +p^2+1 \right) \right) \right)\nonumber\\
           & \equiv b\left(4\left(p^{2k+2}n + \frac{p^{2k+2}-1}{8}\right)\right)\nonumber\\
           & \equiv b\left(4p^{2k+2}n + \frac{p^{2k+2}-1}{2}\right) \pmod{8} \label{c1.2}.
\end{align}
This completes the proof.
\end{proof}

\begin{proof}[Proof of Corollary \ref{C2}]
Substituting $pn+\ell$ into \eqref{c1.2} and then applying \eqref{c1.1.2}, we obtain \eqref{c1.1.1.2}.
\end{proof}

\begin{proof}[Proof of Theorem \ref{Thm 1.3}]
Reducing \eqref{Gen P_B 6n+4} modulo $4$, we have
\begin{align*}
    \sum_{n \geq 0} b(6n+4)q^n & \equiv 9f_6^3 \pmod{36}.
\end{align*}
Now using \eqref{f_1^3 p d}, we get 
\begin{align}\label{mod 36}
    \sum_{n \geq 0} b(6n+4)q^n & \equiv  \frac{9}{2} \sum_{\substack{k=0 \\ k \neq \frac{p-1}{2}}}^{p-1} (-1)^k q^{3k^2+3k} \sum_{n=-\infty}^{\infty}(-1)^n (2pn + 2k + 1) q^{3pn(pn + 2k + 1)} \nonumber\\
    &+ 9(-1)^{\frac{p-1}{2}} pq^{\frac{3(p^2 - 1)}{4}} f(-q^{6p^2})^3 \pmod{36}.
\end{align}
Since $\ell$ was chosen such that $12\ell + 9$ is a quadratic nonresidue modulo $p$, it follows that for any integer $k$,
\[
12pn + 12\ell + 9 \neq (6k + 3)^2.
\]
This implies that $pn + \ell$ cannot be represented as $3k^2+3k$ for any integer $k$. Also, since $12\ell + 9$ is a quadratic nonresidue modulo $p$, we have
\[
12pn + 12\ell + 9 \neq 9p^2.
\]
Hence, $pn + \ell$ cannot be equal to $\frac{3(p^2 - 1)}{4}$ either. Therefore, upon extracting the terms involving $q^{pn + \ell}$, we obtain
\begin{equation}\label{c1.1.3}
    b(6(pn + \ell)+4) \equiv 0 \pmod{36}.
\end{equation}

Further, extracting the terms involving $q^{p^2n+3(p^2-1)/4}$ from both sides of \eqref{mod 36}, we get
\begin{align}\label{mod 36 4}
    b\left( 6 \left(p^2n+\frac{3(p^2-1)}{4}\right) +4 \right) & \equiv 9(-1)^{\frac{p-1}{2}}p f_6^3 \pmod{36},
\end{align}
which is equivalent to
\begin{align}\label{mod 36 5}
     b(6n+4) & \equiv 9^{-1}(-1)^{\frac{1-p}{2}}p^{-1}b\left( 6 \left(p^2n+\frac{3(p^2-1)}{4}\right)+4 \right) \pmod{36}.
\end{align}
Successive iterations on \eqref{mod 36 5} give
\begin{align}
    b(6n+4) &\equiv 9^{-1}(-1)^{\frac{1-p}{2}}p^{-1} b\left( 6 \left(p^2n+\frac{3(p^2-1)}{4}\right)+4 \right) \nonumber\\
           &\equiv 9^{-2}(-1)^{1-p}p^{-2} b\left( 6 \left(p^2\left(p^2n+\frac{3(p^2-1)}{4}\right)+\frac{3(p^2-1)}{4}\right) +4 \right)\nonumber\\
           &\,\,\, \vdots\nonumber\\
           &\equiv 9^{-k-1}(-1)^{\frac{(k+1)(1-p)}{2}}p^{-k-1} b\left( 6\left( p^{2k+2}n + \frac{3(p^2-1)}{4}\left( p^{2k}+ \cdots +p^2+1 \right) \right) +4 \right)\nonumber\\
           & \equiv 9^{-k-1}(-1)^{\frac{(k+1)(1-p)}{2}}p^{-k-1} b\left(6\left(p^{2k+2}n + \frac{3(p^{2k+2}-1)}{4}\right)+4\right)\nonumber\\
           & \equiv 9^{-k-1}(-1)^{\frac{(k+1)(1-p)}{2}}p^{-k-1} b\left(6p^{2k+2}n + \frac{9p^{2k+2}-1}{2}\right) \pmod{36}. \label{c1.3}
\end{align}
This completes the proof.
\end{proof}

\begin{proof}[Proof of Corollary \ref{cor last}]
Substituting $pn+\ell$ into \eqref{c1.3} and then applying \eqref{c1.1.3}, we obtain \eqref{c1.1.1.3}.
\end{proof}

\section{Proof of Theorem \ref{Theorem f_1^7}}\label{sec 5}
Under modulo 9, \eqref{f1} becomes
\begin{equation*}
        \sum_{n \geq 0}  b(6n+3) q^n \equiv 6f_2^7 \pmod{54}.
\end{equation*}
Extracting the terms containing $q^{2n}$ from both sides, we obtain  
\begin{equation*}
        \sum_{n \geq 0}  b(12n+3) q^n \equiv 6f_1^7 = 6\sum_{n=0}^{\infty}a_7(n)q^n \pmod{54},
\end{equation*}
which implies that for all $n\geq 0$,
\begin{equation}\label{f1.1}
    b(12n+3) \equiv 6a_7(n) \pmod{54}.
\end{equation}

We now deduce from  Table \ref{table1} that
\begin{equation}\label{f1.2}
    r_1(7,9) = 5, \qquad r_2(7,9) = 8, \qquad r_3(7,9) = 8.
\end{equation}
In view of \eqref{f_1 7 1}--\eqref{f_1 7 3} and \eqref{f1.2},
\begin{align}
&a_7\left(5^{12\alpha + 11}\left(5n + \beta_1 - 4\right) + \frac{7(5^{12\alpha + 12} - 1)}{24}\right) \nonumber \\
&\equiv a_7\left(7^{18\alpha + 17}\left(7n + \beta_2 - 6\right) + \frac{7(7^{18\alpha + 18} - 1)}{24}\right) \nonumber \\
&\equiv a_7\left(13^{18\alpha + 17}\left(13n + \beta_3 - 12\right) + \frac{7(13^{18\alpha + 18} - 1)}{24}\right) \equiv 0 \pmod{9}, \label{f1.3}
\end{align}
where $0 \leq \beta_1 \leq 4$, $\beta_1 \neq 1$, $0 \leq \beta_2 \leq 6$, $\beta_2 \neq 2$, and $0 \leq \beta_3 \leq 12$, $\beta_3 \neq 3$. 
Congruence \eqref{f1_3} follows from \eqref{f1.1} and \eqref{f1.3}.

By \eqref{f_1 7 4}, \eqref{f1.2} and the fact that $a_7(12) \equiv 0 \pmod{9}$,
\begin{align}
a_7\left( \frac{295 \times 5^{12\alpha} - 7}{24} \right) 
\equiv a_7\left( \frac{295 \times 7^{18\alpha} - 7}{24} \right) 
\equiv a_7\left( \frac{295 \times 13^{18\alpha} - 7}{24} \right) 
\equiv 0 \pmod{9}. \label{f1.4}
\end{align}
Congruence \eqref{f1_4} follows from \eqref{f1.1} and \eqref{f1.4}.

Moreover, it is easy to verify that
\begin{equation}\label{e1}
    A_7(12) \equiv 1 \pmod{9}, \quad C_7(12) \equiv -1 \pmod{9}, \quad E_7(12) \equiv 1 \pmod{9}.
\end{equation}
Thanks to \eqref{f_1 7 5}--\eqref{f_1 7 7}, \eqref{f1.2}, and \eqref{e1}, we find that
\begin{align}\label{e2}
    a_7\left( \frac{7(5^{12\alpha} - 1)}{24} \right) &\equiv a_7\left( \frac{7(13^{18\alpha} - 1)}{24} \right) \equiv 1 \pmod{9},
\\
\label{e3}
    a_7\left( \frac{7(7^{18\alpha} - 1)}{24} \right) &\equiv (-1)^{\alpha} \pmod{9}.
\end{align}
Congruences \eqref{f1_5} and \eqref{f1_6} follow from \eqref{f1.1}, \eqref{e2} and \eqref{e3}. This completes the proof. \qed

\section{Proofs of Theorems \ref{nt1}--\ref{nt2}}\label{sec 6}
Under modulo $8$, \eqref{Gen P_B 6n+4} becomes
\begin{equation}\label{ez1}
    \sum_{n \geq 0}  b(6n+4) q^n \equiv 9f_3^4f_6 \pmod{72}.
\end{equation}
Extracting the terms containing $q^{3n}$ from both sides of \eqref{ez1}, we get  
\begin{equation}\label{ez2.0}
    \sum_{n \geq 0}  b(18n+4) q^n \equiv 9f_1^4f_2 = 9\sum_{n = 0}^\infty a(n)q^n \pmod{72}.
\end{equation}
To prove Theorem \ref{nt1}, we require some lemmas.
\begin{lemma}\label{l200}
    For all $n, k \geq 0$ and primes $p \geq 3$, we have
    \begin{equation}\label{l0.2.1.0}
        a\left(p^{2k}n + \frac{p^{2k} - 1}{4} \right) = U_k(p) a\left(p^2 n + \frac{p^2 - 1}{4} \right) + V_k(p)a(n),
    \end{equation}
    where $a(n)$ is defined by \eqref{n1} and the sequence terms $U_k(p)$ and $V_k(p)$ satisfy
    \begin{align}
        U_k(p) &= rU_{k-1}(p) - sU_{k-2}(p), \label{l0.2.1.1}\\
        V_k(p) &= rV_{k-1}(p) - sV_{k-2}(p) \label{l0.2.1.2},
    \end{align}
     where $r$ and $s$ are defined in \eqref{n4} and
    \begin{equation}\label{l0.2.1.3}
        U_1(p) = 1, \quad U_0(p) = 0, \quad V_1(p) = 0, \quad V_0(p) = 1.
    \end{equation}
\end{lemma}

\begin{proof}
We prove this lemma by induction on $k$. Note that this lemma is true when $k = 0$ and $k = 1$ by \eqref{l0.2.1.3}. Now suppose that Lemma \ref{l200} holds when $k = m$ and $ m + 1$, which gives that
\begin{equation}\label{l0.2.1.4}
a\left(p^{2m}n + \frac{p^{2m} - 1}{4} \right) = U_m(p)a\left(p^2 n + \frac{p^2 - 1}{4} \right) + V_m(p)a(n)
\end{equation}
and
\begin{equation}\label{l0.2.1.5}
a\left(p^{2{m+1}}n + \frac{p^{2{m+1}} - 1}{4} \right) = U_{m+1}(p)a\left(p^2 n + \frac{p^2 - 1}{4} \right) + V_{m+1}(p)a(n).
\end{equation}

In \cite{newman}, Newman proved the following identity on $a(n)$:
\begin{equation}\label{l0.2.1.6}
a\left(p^2 n + \frac{p^2 - 1}{4} \right) = \gamma(n)a(n)-p^3a\left(\frac{n - \frac{p^2 - 1}{4}}{p^2} \right),
\end{equation}
where
\begin{equation*}
    \gamma(n) := a\left(\frac{p^2 - 1}{4} \right) + (-1)^{\frac{p-1}{2}} p \left( \left(\frac{\frac{p^2 - 1}{4}}{p} \right)_L - \left( \frac{\frac{p^2 - 1}{4}-n}{p} \right)_L\right),
\end{equation*}
If we replace $n$ by $p^2 n + \frac{(p^2 - 1)}{4}$ in \eqref{l0.2.1.6}, we obtain
\begin{equation}\label{l0.2.1.7}
a\left(p^4 n + \frac{p^4 - 1}{4} \right) = ra\left(p^2 n + \frac{p^2 - 1}{4} \right) - sa(n).
\end{equation}
Replacing $n$ by $p^{2m} n + \frac{(p^{2m} - 1)}{4}$ in \eqref{l0.2.1.7} and utilizing \eqref{l0.2.1.4} and \eqref{l0.2.1.5} yields
\begin{align*}
    a\left(p^{2m+2}n + \frac{p^{2m+2} - 1}{4} \right)
&= ra\left(p^{2m+2}n + \frac{p^{2m} - 1}{4} \right) - sa\left(p^{2m}n + \frac{p^{2m} - 1}{4} \right) \\
&= r\left(U_{m+1}(p) a\left(p^2 n + \frac{p^2 - 1}{4} \right) + V_{m+1}(p)a(n)\right) \\
&\quad - s\left(U_m(p)a\left(p^2 n + \frac{p^2 - 1}{4} \right) + V_m(p)a(n)\right) \\
&= (r U_{m+1}(p) - s U_m(p)) a\left(p^2 n + \frac{p^2 - 1}{4} \right) + (r V_{m+1}(p) - s V_m(p)) a(n) \\
&= U_{m+2}(p) a\left(p^2 n + \frac{p^2 - 1}{4} \right) + V_{m+2}(p)a(n),
\end{align*}
which implies that \eqref{l0.2.1.0} holds when $k = m + 2$. Thus Lemma \ref{l200} is proved by induction.
\end{proof}

\begin{lemma}\label{l201}
    For all $k \geq 0$, we have
    \begin{equation}\label{l0.2.1.8}
rU_{\nu(p)k+\nu(p)-1}(p) + V_{\nu(p)k+\nu(p)-1}(p) \equiv 0 \pmod{8},
\end{equation}
where $\nu(p)$ is defined by \eqref{n2} and $U_k(p)$ and $V_k(p)$ are defined by \eqref{l0.2.1.1} and \eqref{l0.2.1.2}, respectively.
\end{lemma}
\begin{proof}
    Lemma \ref{l201} will be proved by induction on $k$. In view of \eqref{l0.2.1.1}--\eqref{l0.2.1.3}, it is easy to check that
\begin{equation}
rU_{\nu(p)-1}(p) + V_{\nu(p)-1}(p) = h(p),\label{l0.2.1.9}
\end{equation}
where
\begin{equation}\label{l0.2.1.10}
h(p) := 
\begin{cases}
r, & \text{if} \hspace{1.5mm} \nu(p) = 2, \\
r^3 - 2rs, & \text{if} \hspace{1.5mm} \nu(p) = 4, \\
r^7 - 6r^5s + 10s^2r^3 - 4rs^3, & \text{if} \hspace{1.5mm} \nu(p) = 8.
\end{cases}
\end{equation}
Based on \eqref{n2} and \eqref{l0.2.1.10}, one can check that
\begin{equation}\label{l0.2.1.11}
h(p) \equiv 0 \pmod{8}.
\end{equation}
So \eqref{l0.2.1.8} is true when $k = 0$. Suppose that \eqref{l0.2.1.8} holds when $k = m$ ($m \geq 0$) which implies that
\begin{equation}\label{l0.2.1.12}
r U_{\nu(p)m+\nu(p)-1}(p) + V_{\nu(p)m+\nu(p)-1}(p) \equiv 0 \pmod{8}.
\end{equation}
In light of \eqref{l0.2.1.1}, \eqref{l0.2.1.2} and the values of $\nu(p)$, we have
\begin{align}
r U_{\nu(p)m+2\nu(p)-1}(p) + V_{\nu(p)m+2\nu(p)-1}(p)
&= h(p)\left(r U_{\nu(p)m+\nu(p)}(p) + V_{\nu(p)m+\nu(p)}(p)\right) \nonumber \\
&\quad + g(p)\left(r U_{\nu(p)m+\nu(p)-1}(p) + V_{\nu(p)m+\nu(p)-1}(p)\right),\label{l0.2.1.13}
\end{align}
where $g(p)$ and $h(p)$ are defined by \eqref{n3} and \eqref{l0.2.1.10}, respectively. 

Thanks to \eqref{l0.2.1.11}, \eqref{l0.2.1.12}, and \eqref{l0.2.1.13}, we get
\begin{equation*}
    r U_{\nu(p)m+2\nu(p)-1}(p) + V_{\nu(p)m+2\nu(p)-1}(p) \equiv 0 \pmod{8},
\end{equation*}
which implies \eqref{l0.2.1.8} when $k = m + 1$. So Lemma \ref{l201} is also proved by induction.
\end{proof}

\begin{lemma}\label{l202}
    For all $k \geq 0$, we have
\begin{align}\label{l0.2.1.14}
U_{\nu(p)k}(p) &\equiv 0 \pmod{8}
\\
\label{l0.2.1.15}
V_{\nu(p)k}(p) &\equiv g(p)^k \pmod{8},
\end{align}
    where $\nu(p)$ and $g(p)$ are defined by \eqref{n2} and \eqref{n3}, respectively.
\end{lemma}
\begin{proof}
Once again, we employ induction to prove \eqref{l0.2.1.14} and \eqref{l0.2.1.15}. Note that \eqref{l0.2.1.14} and \eqref{l0.2.1.15} hold when $k = 0$ because $U_0(p) = 0$ and $V_0(p) = 1$. Suppose that \eqref{l0.2.1.14} and \eqref{l0.2.1.15} hold when $k = m$, that means
\begin{align}\label{l0.2.1.16}
U_{\nu(p)m}(p) &\equiv 0 \pmod{8}
\\
\label{l0.2.1.17}
V_{\nu(p)m}(p) &\equiv g(p)^m \pmod{8}.
\end{align}
According to the the values of $\nu(p)$, \eqref{l0.2.1.1} and \eqref{l0.2.1.2}, we have
\begin{equation}\label{l0.2.1.18}
U_{\nu(p)m+\nu(p)}(p) = h(p) U_{\nu(p)m+1}(p) + g(p) U_{\nu(p)m}(p),
\end{equation}
\begin{equation}\label{l0.2.1.19}
V_{\nu(p)m+\nu(p)}(p) = h(p) V_{\nu(p)m+1}(p) + g(p) V_{\nu(p)m}(p),
\end{equation}
where $g(p)$ and $h(p)$ are defined by \eqref{n3} and \eqref{l0.2.1.10}, respectively. By \eqref{l0.2.1.11}, \eqref{l0.2.1.16}, and \eqref{l0.2.1.18}, we deduce that \eqref{l0.2.1.14} is true when $k = m + 1$. It follows from \eqref{l0.2.1.11}, \eqref{l0.2.1.17} and \eqref{l0.2.1.19} that
\begin{equation*}
    V_{\nu(p)m+\nu(p)}(p) \equiv g(p)^{m+1} \pmod{8},
\end{equation*}
which implies that \eqref{l0.2.1.15} holds when $k = m + 1$. Thus, Lemma \ref{l202} is proved by induction.
\end{proof}

Now, we are in a position to prove Theorem \ref{nt1}.

\begin{proof}[Proof of Theorem \ref{nt1}]
Substituting \eqref{l0.2.1.6} into \eqref{l0.2.1.0} yields
\begin{align}\label{l0.2.1.20}
a\left(p^{2k}n + \frac{p^{2k} - 1}{4}\right)
&= (U_k(p)(n) + V_k(p)) a(n) - s U_k(p) a\left(n - \frac{p^2 - 1}{4} \right).
\end{align}
Replacing $n$ by $pn + \frac{p^2 - 1}{4}$ in \eqref{l0.2.1.20}, we get
\begin{align}\label{l0.2.1.21}
a\left(p^{2k+1}n + \frac{p^{2k+2} - 1}{4}\right)
&= (r U_k(p) + V_k(p)) a\left(pn + \frac{p^2 - 1}{4} \right) - s U_k(p) a\left( \frac{n}{p} \right),
\end{align}
where $r$ and $s$ are defined by \eqref{n4}. 

Replacing $k$ by $\nu(p)k + \nu(p) - 1$ in \eqref{l0.2.1.21} and using \eqref{l0.2.1.18} gives
\begin{align}
a\left(p^{2\nu(p)(k+1)-1}n + \frac{p^{2\nu(p)(k+1)} - 1}{4}\right)
&= -s U_{\nu(p)k+\nu(p)-1}(p) a\left( \frac{n}{p} \right) \pmod{8},
\end{align}
which implies that if $p \nmid n$, then
\begin{align}\label{l0.2.1.22}
a\left(p^{2\nu(p)(k+1)-1}n + \frac{p^{2\nu(p)(k+1)} - 1}{4}\right) \equiv 0 \pmod{8}.
\end{align}
By \eqref{ez2.0},
\begin{align}\label{l0.2.1.23}
b(18n+4) \equiv 9a(n) \pmod{72},
\end{align}
Congruence \eqref{newm1} follows after replacing $n$ by $p^{2\nu(p)(k+1)-1}n + \frac{p^{2\nu(p)(k+1)} - 1}{4}$ in \eqref{l0.2.1.23} and utilizing \eqref{l0.2.1.22}.

Replacing $k$ by $\nu(p)k$ in \eqref{l0.2.1.0} and utilizing \eqref{l0.2.1.14} and \eqref{l0.2.1.15}, we get
\begin{align}\label{l0.2.1.25}
a\left(p^{2\nu(p)k}n + \frac{p^{2\nu(p)k} - 1}{4}\right) \equiv g(p)^k a(n) \pmod{8},
\end{align}
where $g(p)$ is defined by \eqref{n3}. Setting $n = 1$ in \eqref{l0.2.1.25} yields
\begin{align}\label{l0.2.1.26}
a\left(p^{2\nu(p)k} + \frac{p^{2\nu(p)k} - 1}{4}\right) \equiv 4g(p)^k \pmod{8}.
\end{align}
Replacing $n$ by $p^{2\nu(p)k} + \frac{p^{2\nu(p)k} - 1}{4}$ in \eqref{l0.2.1.23} and using \eqref{l0.2.1.26}, we arrive at \eqref{newm2}. This completes the proof.
\end{proof}

\begin{proof}[Proof of Theorem \ref{nt3}]

We verified using Radu's algorithm \cite{Rad09} that \( a(6n+4) \equiv 0 \pmod{8} \) for all \( n \geq 0 \). Replacing \( n \) with \( 6n+4 \) in \eqref{l0.2.1.25} then gives
\begin{align}\label{l0.2.1.26.0}
a\left(6p^{2\nu(p)k}n + \frac{17p^{2\nu(p)k} - 1}{4}\right) \equiv 0 \pmod{8}.
\end{align}
Replacing $n$ by $6p^{2\nu(p)k}n + \frac{17p^{2\nu(p)k} - 1}{4}$ in \eqref{l0.2.1.23} and using \eqref{l0.2.1.26.0}, we arrive at \eqref{newm2.0}. This completes the proof.
\end{proof}

\begin{proof}[Proof of Theorem \ref{nt2}]
We prove Theorem \ref{nt2} by induction on the total number of prime factors of $n$. Suppose that $m$ is a nonnegative integer such that $a(m) \equiv 0 \pmod{8}$. By \eqref{l0.2.1.23}, we have
\begin{equation*}
    b(18n+4) \equiv 9a(n) \pmod{72},
\end{equation*}
which implies \eqref{ez2} holds when $n = 1$ ($n$ has no prime factors).

Suppose that the prime factorization of $4m + 1$ is $4m + 1 = \prod_{j=1}^t h_j \prod_{j=1}^t g_j^{\alpha_j}$ with each $\alpha_j \ge 2$. Let $p_1 \ge 3$ be a prime with $(p_1, \prod_{j=1}^t g_j^{\alpha_j}) = 1$. Replacing $(n, p)$ by $(m, p_1)$ in \eqref{l0.2.1.6} and utilizing the hypothesis that $a(m) \equiv 0 \pmod{8}$ yields
\begin{equation}\label{l0.2.1.27}
a\left(mp_1^2 + \frac{p_1^2 - 1}{4}\right) \equiv -sa \left( m - \frac{p_1^2 - 1}{4p_1^2} \right) \pmod{8}.
\end{equation}
It is easy to check that
\begin{equation*}
\frac{m - \frac{p_1^2 - 1}{4}}{p_1^2} = \frac{4m + 1 - p_1^2}{4p_1^2} = \frac{\prod_{i=1}^t h_i \prod_{j=1}^t g_j^{\alpha_j} - p_1^2}{4p_1^2}
\end{equation*}
is not an integer since $\gcd(p_1, \prod_{j=1}^t g_j^{\alpha_j}) = 1$. Thus,
\begin{equation}\label{l0.2.1.28}
a\left( m - \frac{p_1^2 - 1}{4p_1^2} \right) = 0.
\end{equation}
Thanks to \eqref{l0.2.1.27} and \eqref{l0.2.1.28}, we obtain
\begin{equation*}
    a\left( mp_1^2 + \frac{p_1^2 - 1}{4} \right) \equiv 0 \pmod{8},
\end{equation*}
from which with \eqref{l0.2.1.23}, we see that
\begin{equation*}
    b\left( 18mp_1^2+\frac{9p_1^2-1}{2} \right)\equiv 0 \pmod{72}.
\end{equation*}
Therefore, \eqref{ez2} holds when $n = p_1$ ($n$ has only one prime factor).
Now assume that \eqref{ez2} holds for all integers with no more than $k$ prime factors. To prove Theorem \ref{nt2}, it suffices to show that \eqref{ez2} holds when $n$ has $k + 1$ prime factors. Write $n$ as $n = p_1 p_2 \cdots p_k p_{k+1}$ where $3 \le p_1 \le p_2 \le \cdots \le p_k < p_{k+1}$ with $(p_1 \cdots p_{k} p_{k+1}, \prod_{j=1}^t g_j^{\alpha_j}) = 1$. By \eqref{l0.2.1.23} and the hypothesis that \eqref{ez2} holds for all integers with no more than $k$ prime factors, we have
\begin{align}\label{l0.2.1.29}
a\left( mp_1^2 p_2^2 \cdots p_{k-1}^2 + \frac{p_1^2 p_2^2 \cdots p_{k-1}^2 - 1}{4} \right) &= 9b\left( 18mp_1^2 p_2^2 \cdots p_{k-1}^2+\frac{9p_1^2 p_2^2 \cdots p_{k-1}^2-1}{2} \right)\notag\\
&\equiv 0 \pmod{72}
\end{align}
and
\begin{align}\label{l0.2.1.30}
&a\left(mp_1^2 p_2^2 \cdots p_{k-1}^2p_{k+1}^2 + \frac{p_1^2 p_2^2 \cdots p_{k-1}^2p_{k+1}^2 - 1}{4} \right)\notag\\
&\equiv 9b\left( 18mp_1^2 p_2^2 \cdots p_{k-1}^2p_{k+1}^2+\frac{9p_1^2 p_2^2 \cdots p_{k-1}^2p_{k+1}^2-1}{2} \right) \notag\\
&\equiv 0 \pmod{72}.
\end{align}
If we replace $(n,p)$ by $\left(mp_1^2 p_2^2 \cdots p_{k-1}^2p_{k+1}^2 + \frac{p_1^2 p_2^2 \cdots p_{k-1}^2p_{k+1}^2 - 1}{4}, p_{k+1}\right)$ in \eqref{l0.2.1.6}, and utilize \eqref{l0.2.1.30}, we have
\begin{align}\label{l0.2.1.31}
a&\left( mp_1^2 p_2^2 \cdots p_{k-1}^2p_k^2 p_{k+1}^2 + \frac{p_1^2 p_2^2 \cdots p_{k-1}^2p_k^2 p_{k+1}^2 - 1}{4} \right) \notag\\
&\equiv -sa\left(\frac{mp_1^2 \cdots p_{k-1}^2p_k^2 + \frac{p_1^2 \cdots p_{k-1}^2p_k^2 - 1}{4}}{p_{k+1}^2}\right) \pmod{8}.
\end{align}
If $p_{k+1} = p_k$, then \eqref{l0.2.1.31} can be rewritten as
\begin{align}\label{l0.2.1.32}
a\left(mp_1^2 p_2^2 \cdots p_k^2 p_{k+1}^2 + \frac{p_1^2 p_2^2 \cdots p_k^2 p_{k+1}^2 - 1}{4} \right) 
&\equiv -sa\left(mp_1^2 \cdots p_k^2p_{k+1}^2 + \frac{p_1^2 \cdots p_k^2p_{k+1}^2 - 1}{4}\right) \notag\\
&\equiv 0 \pmod{8}. \qquad \text{(by \eqref{l0.2.1.29})}
\end{align}
If $p_{k+1} > p_k$, then $p_{k+1} \notin \{p_1, p_2, \ldots, p_k\}$. Note that
\begin{align*}
\frac{mp_1^2 p_2^2 \cdots p_{k-1}^2p_{k}^2 + \frac{p_1^2 \cdots p_{k-1}^2p_{k}^2 - p_{k+1}^2}{4}}{p_{k+1}^2} &= \frac{(4m + 1)p_1^2 \cdots p_k^2 - p_{k+1}^2 - p_{k+1}^2}{4p_{k+1}^2}\\
&
= \frac{p_1^2 \cdots p_k^2 \prod h_i \prod g_j^{\alpha_j} - p_{k+1}^2}{4p_{k+1}^2}
\end{align*}
is not an integer since $\left(p_{k+1}, \prod g_j^{\alpha_j} \right) = 1$ and $\left(p_{k+1}, p_1 \cdots p_k \right) = 1$. Thus,
\begin{equation}\label{l0.2.1.33}
a\left( \frac{mp_1^2 p_2^2 \cdots p_{k-1}^2p_k^2 + \frac{p_1^2 \cdots p_{k-1}^2p_k^2 - p_{k+1}^2}{4}}{p_{k+1}^2} \right) = 0.
\end{equation}

In light of \eqref{l0.2.1.31}--\eqref{l0.2.1.33}, we get
\begin{equation*}
    a\left( mp_1^2 p_2^2 \cdots p_k^2 p_{k+1}^2 + \frac{p_1^2 p_2^2 \cdots p_k^2 p_{k+1}^2 - 1}{4} \right) \equiv 0 \pmod{8},
\end{equation*}
from which with \eqref{l0.2.1.23}, we deduce that \eqref{ez2} holds when $n = p_1 p_2 \cdots p_k p_{k+1}$. Therefore, Theorem \ref{nt2} is proved by induction. This completes the proof.
\end{proof}

\section{Concluding Remarks}\label{sec 7}
\begin{itemize}
    
    \item[1. ] An analytic or combinatorial proof of \eqref{Gen P_B 3n} is desirable.  
    \item[2. ] It appears that the coefficients of $\mathcal{B}(q)$ are rich in arithmetic properties. Surely, many more congruences for $\mathcal{B}(q)$ are waiting to be discovered. For example, the following individual congruences can be readily verified using Radu's algorithm \cite{Rad09}. However, it would be interesting to find their $q$-series proofs. For all $n \geq 0$, we have  
\begin{align*}        
        b\left(30n+r\right)&\equiv 0 \pmod{20},\quad r\in\{6,18\},\\
        b\left(42n+r\right)&\equiv 0 \pmod{4},\quad r\in\{6,30,36\},\\
        b\left(42n+r\right)&\equiv 0 \pmod{7},\quad r\in\{21,33,39\},\\
        b\left(60n+r\right)&\equiv 0 \pmod{8},\quad r\in\{36,48\},\\
        b\left(84n+r\right)&\equiv 0 \pmod{8},\quad r\in\{36,48,72\},\\
        b\left(150n+r\right)&\equiv 0 \pmod{16},\quad r\in\{72,108\},\\
        b\left(300n+r\right)&\equiv 0 \pmod{25},\quad r\in\{177,237,297\}.
\end{align*}
\item[3. ] Using Radu's algorithm, one can verify that the following congruences hold for all \( n \geq 0 \)
\begin{align*}
    a(9n+5) &\equiv  a(9n+8) \equiv 0 \pmod{8},\\
    a(10n+4) &\equiv a(10n+8) \equiv 0 \pmod{8},\\
    a(14n+4) &\equiv  a(14n+8) \equiv  a(14n+10) \equiv 0 \pmod{8}
\end{align*}
and so on. Therefore, if we replace \( n \) with an arithmetic progression \( An + B \) such that \( a(n) \equiv 0 \pmod{8} \) in \eqref{l0.2.1.25}, and then substitute the resulting arithmetic progression into \eqref{l0.2.1.23}, we obtain several congruences for \( b(n) \) modulo \( 72 \).
\item[4. ] As a natural extension, one may find infinite families for $b(n)$
modulo \(162\) and higher moduli. We leave this task to the interested reader.
\item[5. ] We state the following conjecture.
\begin{conjecture}\label{Con1}
Suppose that $m$ is a positive integer such that $a(m) \equiv 0 \pmod{8}$, where $a(m)$ is defined in \eqref{n1}. Let $4m + 3 = \prod_{i=1}^{u} h_i \prod_{j=1}^{v} g_j^{\alpha_j}$ with each $\alpha_j \geq 2$ is the prime factorization of $4m + 3$. Then for $n \geq 1$, we have
\begin{equation*}\label{ez3}
b\left( 18mn^2+\frac{9n^2-1}{2} \right)\equiv 0 \pmod{72},
\end{equation*}
where $\left(n, 2 \prod_{j=1}^{v} g_j^{\alpha_j} \right) = 1$.
\end{conjecture}
\item[6. ] We conclude the paper with the following $q$-product identities analogous to \eqref{Gen P_B 3n} for the remaining second order mock theta functions, obtained using the \texttt{qseries} package in Maple. An interested reader may attempt to provide analytic proofs of these.
\begin{align*}
    \sum_{n \geq 0} \mathcal{A}_2(3n+1) q^n &= \frac{f_2^4 f_3^2 f_4}{f_1^5 f_6}, \\
    \sum_{n \geq 0} \mu_2(3n+1) q^n &= -\frac{f_2^7 f_{12}^2}{f_1 f_4^6 f_6},
\end{align*}
where 
\begin{align*}
\mathcal{A}(q) &=\sum_{n \geq 0} \frac{q^{(n+1)^2}(-q;q^2)_n}{(q;q^2)_{n+1}^2} = \sum_{n \geq 0} \frac{q^{n+1}(-q^2;q^2)_n}{(q;q^2)_{n+1}}=\sum_{n\ge0}\mathcal{A}_2(n)q^n,\\
    \mu(q) &:= \sum_{n \geq 0} \frac{(-1)^n q^{n^2} (q;q^2)_n}{(-q^2;q^2)_n^2}=\sum_{n\ge0}\mu_2(n)q^n.
\end{align*}
\end{itemize}

\end{document}